\pdfoutput=1

\documentclass[11pt,letterpaper]{article}
\usepackage{amssymb}
\usepackage{amsthm}
\usepackage{amsmath}
\usepackage{fullpage}
\usepackage{mathtools}
\usepackage{graphicx,url,subfigure,bbm}
\usepackage{hyperref}
\usepackage{pdfsync}
\newtheorem{theorem}{Theorem}[section]
\newtheorem{lemma}[theorem]{Lemma}

\newtheorem{corollary}[theorem]{Corollary}

\newtheorem{fact}[theorem]{Fact}

\newtheorem{remark}{Remark}[section]

\newtheorem{conjecture}[theorem]{Conjecture}

\newcommand{\pr}{\mathbb{P}}

\newcommand{\E}{{\mathbb E}}

\newcommand{\cP}{\mathcal{P}}

\newcommand{\supp}{\mathrm{supp}}

\newcommand{\remove}[1]{}
\newcommand{\1}{\mathbf{1}}

\newcommand{\e}{\varepsilon}
\newcommand{\eps}{\varepsilon}

\newcommand{\Id}{\mathrm{Id}}
\newcommand{\RR}{\mathbb{R}}

\newcommand{\grad}{\nabla}

\newcommand*{\defeq}{\mathrel{\vcenter{\baselineskip0.5ex \lineskiplimit0pt
                     \hbox{\scriptsize.}\hbox{\scriptsize.}}}%
                     =}

\def\P{\mathbb{P}}

\begin{document}

\title{{\bf Regularization under diffusion and \\ anti-concentration of the information content}}
\author{Ronen Eldan\footnote{Weizmann Institute of Science.} \and James R. Lee\footnote{University of Washington.}}
\date{}
\maketitle

\begin{abstract}
Under the Ornstein-Uhlenbeck semigroup $\{U_t\}$, any
non-negative measurable $f : \mathbb R^n \to \mathbb R_+$
exhibits a uniform tail bound better than that implied by Markov's inequality
and conservation of mass:
For every $\alpha \geq e^3$, and $t > 0$,
\[
   \gamma_n\left(\left\{x \in \mathbb R^n : U_t f(x) > \alpha \int f\,d\gamma_n\right\}\right) \leq C(t) \frac{1}{\alpha} \sqrt{\frac{\log \log \alpha}{\log \alpha}}\,,
\]
where $\gamma_n$ is the $n$-dimensional Gaussian measure and $C(t)$ is
a constant depending only on $t$.
This confirms positively the Gaussian limiting
case of Talagrand's convolution conjecture (1989).

This is shown to follow from a more general phenomenon.
Suppose that $f : \RR^n \to \RR_+$ is {\em semi-log-convex}
in the sense that for some $\beta > 0$, for all $x \in \RR^n$,
the eigenvalues of $\nabla^2 \log f(x)$ are at least $-\beta$.
Then $f$ satisfies a tail bound asymptotically better
than that implied by Markov's inequality.
\end{abstract}

%\tableofcontents

\section{Introduction}

Let $n \geq 1$ and equip $\mathbb R^n$ with the standard Gaussian measure $\gamma_n$.
Consider a function $f : \mathbb R^n \to \mathbb R$ in $L^1(\gamma_n)$.
The Ornstein-Uhlenbeck semi-group $\{U_t : t \geq 0\}$ is defined by
\[
U_t f(x) = \E f\left(e^{-t} x + \sqrt{1-e^{-2t}} Z\right)\,,
\]
where $Z$ has law $\gamma_n$.
One expects that the action of such a diffusion process serves to smoothen $f$.
Indeed, Nelson's hypercontractivity theorem \cite{Nelson73} shows that $U_t$
is a contraction from $L^p(\gamma_n)$ to $L^q(\gamma_n)$ for $1 < p \leq q$
and $t \geq \frac12 \log \frac{q-1}{p-1}$.

The concept of hypercontractivity plays an important role in several mathematical fields.
For example, in quantum field theory hypercontractivity can
often be used to show that a Hamiltonian is essentially self-adjoint on its
domain, laying the foundation for various constructions (see, e.g., \cite{GRS75}).
We refer to the surveys \cite{DLS92,Gross06}.
In the study of partial differential equations, it is a key method in several approaches to establishing the
existence and uniqueness of smooth solutions to evolution equations \cite{B11}.
Hypercontractivity is also a basic tool in establishing superconcentration \cite{Chatterjee14}.

In the present work, we assert a regularizing effect of $U_t$ merely
assuming that $f \in L^1(\gamma_n)$.  An important special case
is when $f$ is simply the indicator of a measurable subset of $\mathbb R^n$.
Assume now that $f : \mathbb R^n \to \mathbb R_+$
is non-negative.  Certainly we have Markov's inequality:
For any $\alpha \geq 1$,
\[
\gamma_n\left(\vphantom{\bigoplus}\{x : f(x) \geq \alpha \|f\|_1 \}\right) \leq \frac{1}{\alpha}\,,
\]
where we use $\|f\|_1 = \int |f|\,d\gamma_n$.
Of course, this bound is easily seen to be tight for any $\alpha > 0$ by taking
$f=\1_S$ for a measurable subset $S \subseteq \mathbb R^n$
with $\gamma_n(S)=1/\alpha$.  The mass of $f=\1_S$ lies on a single level set.
A very natural question arises:
Can a smoothed version of $f$, i.e. $U_t f$ for some $t > 0$, have a non-negligible fraction
of its mass near a single large value?
Talagrand conjectured that this cannot be the case.\footnote{Talagrand actually made a stronger conjecture \cite{T89}
that a similar statement should hold in the discrete cube.  We refer the reader to Section \ref{sec:discrete}.
We attribute this weaker conjecture to him---with permission---to stress his role in predicting
the phenomenon.}

\begin{conjecture}\label{conj:main}
For every $t > 0$, there exists a function $\psi_t : \mathbb [1,\infty) \to [1,\infty)$ with $\lim_{\alpha \to \infty} \psi_t(\alpha) = \infty$ such
that for any measurable $f : \mathbb R^n \to \mathbb R_+$ and any $\alpha > 1$,
\begin{equation}\label{eq:conj-main}
\gamma_n\left(\vphantom{\bigoplus}\{x : U_t f(x) > \alpha \|f\|_1 \}\right) \leq \frac{1}{\alpha\, \psi_t(\alpha)}\,.
\end{equation}
\end{conjecture}

One should recall here that $U_t$ preserves both positivity and the mean value; for non-negative $f$, we have $\|U_t f\|_1=\|f\|_1$.
The conjecture posits a uniform bound on the tail of the smoothed function.  Talagrand
notes that the best rate of decay one can expect is $\psi_t(\alpha) = c(t) \sqrt{\log \alpha}$
where $c(t)$ is some function depending only on $t$.  We resolve the conjecture
positively (Corollary \ref{cor:main} below) and achieve the bound \[\psi_t(\alpha) = c(t) \sqrt{\frac{\log \alpha}{\log \log \alpha}} \,.\]

Ball, Barthe, Bednorz, Oleszkiewicz, and Wolff \cite{BBBOW13} prove
that Conjecture \ref{conj:main}
holds in any fixed dimension; they achieve $\psi_t(\alpha) = C(t,n) \sqrt{\log \alpha}/(\log \log \alpha)$ where
$C(t,n)$ is a constant depending (exponentially) on the dimension $n$.

\medskip
\noindent
{\bf Anti-concentration of the information content.}
Fix a reference measure $\mu$ on $\RR^n$, and let $X$ be a random vector whose law has density $f d \mu$,
for a non-negative measurable function satisfying $\int f d \mu = 1$.
Following \cite{BM11}, consider the random variable
\begin{equation}\label{eq:info-content}
h(X) \coloneqq \log f(X).
\end{equation}
This quantity is referred to as the {\em information content} of the vector $X$.
In \cite{BM11} it is shown that when the reference measure $\mu$ is the Lebesgue measure and the random vector $X$ is log-concave, the information content is concentrated around its mean, which is the {\em relative entropy of $f$ with respect to $\mu$:}
\[
   H_{\mu}(f) \coloneqq \int f \log f\,d\mu = \E[\log f(X)] = \E[h(X)]\,.
\]

Our goal is to prove that a certain class of densities satisfies a tail bound
stronger than that implied by Markov's inequality.
While such a tail bound is a weak assertion about concentration,
we stress now that the desired improvement over Markov's inequality
is equivalent to establishing a form of {\em anti-concentration}
for the information content.
Indeed, our
verification of Conjecture~\ref{conj:main} will proceed in this manner.

We will show that for every $t > 0$,
the information contents of densities $\{U_t f\}$
are uniformly anti-concentrated in the following sense.
For every $t > 0$, there is a constant $C(t)$ such that for all $y$ sufficiently large
\begin{equation}\label{eq:anti-con}
   \P\left(\vphantom{\bigoplus}\!\left|h(X)-y\right| \leq 1\right) \leq C(t) \sqrt{\frac{\log y}{y}}\,.
\end{equation}
where $X$ is the random vector with law $U_t f d\gamma_n$ and $h(X)=\log U_t f(X)$.
Indeed, this achieves our goal:
If $G$ has law $\gamma_n$, then
\[
   \P\left(U_t f(G) \in [e^j \alpha,e^{j+1} \alpha]\right) \leq \frac{e^{-j}}{\alpha} \P\left(U_t f(X) \in [e^j \alpha,e^{j+1} \alpha]\right) = \frac{e^{-j}}{\alpha} \P\left(|h(X)-(\alpha+j)|\leq 1\right)\,.
\]
Now employing \eqref{eq:anti-con} and
summing over $j \geq 0$ gives \eqref{eq:conj-main} with $\psi_t(\alpha) \asymp \sqrt{\frac{\log \log \alpha}{\log \alpha}}$.

Consider an illustration:  Suppose a runner at time $t=0$ attempts to stop at distance about $y$ from her
starting point at time $t=1$.  Naturally, hitting the target exactly will be difficult,
and as $y$ increases, her accuracy will diminish as her average speed must increase.
One can think of various densities $f$ as stratgies for the runner,
and \eqref{eq:anti-con} asserts a uniform bound on the difficulty of hitting a distant target accurately.

Indeed, the proof proceeds along these lines:  We associate to $h(X)$ a stochastic process $\{Z_t\}$ such that
$Z_0=0$ and $Z_1$ has the law of $h(X)$.  If $Z_1 > y$ for some large $y$, we argue that the process must have
significant ``kinetic energy'' at most times $t \in [0,1]$ (see \eqref{eq:optimality}).  Small pushes at those energetic times do not
change the measure of $\{Z_t\}$ much, but they have a substantial effect on the outcome $Z_1$.
This implies that it is
impossible for $h(X)$ to concentrate near a specific value $y$,
and this difficulty increases with $y$.

\begin{remark}
Observe that \eqref{eq:anti-con} is not written in a more standard form (in terms of the L\'evy
concentration function) only because we did not normalize by the ``kinetic energy.''
If we instead define $\hat{h}(X) \coloneqq \left(\frac{\log U_t f(X)}{\log \log U_t f(X)}\right)_+^{1/2}$,
then Lemma \ref{lem:mass} implies that for every $\e > 0$,
\[
   \P\left(\left|\hat{h}(X)-y\right| < \e\right) \leq C \e + e^{-c y^2} \qquad \forall y \geq 0\,,
\]
where $C=C(t) \geq 1, c=c(t) > 0$ are numbers depending only on $t$.
\end{remark}

% note here that 

%that if we allow the heat to diffuse for a short period of time, then the resulting distribution cannot have a large fraction of its heat concentrated
%in a narrow range of high temperatures.
%More formally, we will see that for $t > 0$,
%\begin{equation}\label{eq:utf}
%\int (U_t f) \1_{\{U_t f \in [\alpha,2 \alpha]\}}\,d \gamma_n \leq c(t) \frac{\sqrt{\log \log \alpha}}{\sqrt{\log \alpha}} \|f\|_1\,.
%\end{equation}

\remove{
\medskip
\noindent
{\bf Many shades of grey.}
There is an appealing intuitive picture behind Conjecture \ref{conj:main}.
If we think of $f : \mathbb R^n \to \mathbb R_+$ as giving greyscale values to ``pixels'' in $\mathbb R^n$
(with larger values closer to white and smaller values closer to black),
then the conjecture asserts that the blurred image can never have most of its ``grey mass'' concentrated
on a small number of very light shades of grey.
More formally, we will see that for $t > 0$,
\[
\int (U_t f) \1_{\{U_t f \in [\alpha,2 \alpha]\}}\,d \gamma_n \leq c(t) \frac{(\log \log \alpha)^{3/2}}{\sqrt{\log \alpha}} \|f\|_1\,.
\]
}

\medskip
\noindent
{\bf An dual perspective: The geometry of small sets.}
A dual point of view is helpful in understanding the geometric
content of Conjecture \ref{conj:main}. Fix $t > 0$,
let $S \subseteq \mathbb R^n$ be an open subset, and consider the set of
non-negative functions
$g : \mathbb R^n \to \mathbb R_+$ supported on $S$, and such that $\|U_t g\|_{\infty} \leq 1$.
Suppose we wish to maximize $\int g\,d\gamma_n$ subject to these constraints.

Clearly the choice $g=\1_S$ has $\int g\,d\gamma_n = \gamma_n(S)$.
Conjecture \ref{conj:main} asserts that there should be a strategy that does asymptotically better:
As $\gamma_n(S) \to 0$, it should be possible to achieve \[\gamma_n(S)^{-1} \int g\,d\gamma_n \to \infty\,.\]

In fact, the largest function $\psi_t$ achievable
in Conjecture \ref{conj:main} is precisely the same as the largest function $\psi_t$ such that
the following holds for every open $S \subseteq \mathbb R^n$:
\begin{equation}\label{eq:dual}
   \sup_{\substack{g : \RR^n \to \RR_+ \\ \supp(g) \subseteq S}}
      \left\{\int g \,d\gamma_n : \|U_t g\|_{\infty} \leq 1 \right\}
  % \sup_{\supp(g) \subseteq S} \frac{1}{\|U_t g\|_{\infty} } 
   \geq \psi_t\left(1/\gamma_n(S)\right)\gamma_n(S)\,.
\end{equation}
This dual characterization is a straightforward consequence of Hahn-Banach and self-adjointness of $U_t$ as an operator on $L^2(\gamma_n)$.
This optimization problem has a certain isoperimetric flavor
because it is intuitive
is that to make $\int g\,d\gamma_n$ significantly larger subject to the constraint $\|U_t g\|_{\infty} \leq 1$,
one should take advantage of the smoothing effects of $U_t$ near the boundary of $S$.

To make this slightly more concrete, consider the case $n=1$.
One can prove Conjecture \ref{conj:main} for $n=1$
via the dual formulation \eqref{eq:dual} as follows:
Given $S \subseteq \mathbb R$, one should choose $g$ to be a Dirac mass near the point
of $\mathbb R \setminus S$ which is closest to the origin.
(Strictly speaking, one should take a sequence of points in $S$ and a sequence of functions
approximating Dirac masses at those points.)  From the value $\gamma_n(S)$, one can conclude that
$S$ contains a point sufficiently close to the origin.
 A simple calculation with the Gaussian density yields the
 desired bound.\footnote{Completion of this sketch is Exercise 11.31 in O'Donnell's book \cite{OD14}.}

 \subsection{Semi-log-convexity and anti-concentration of the information content}

The resolution of Conjecture \ref{conj:main} arises from a more general
phenomenon.  Say that a function $f : \mathbb R^n \to \RR_+$ is {\em $\beta$-semi-log-convex}
if the function $x \mapsto \log f(x) + \frac{\beta}{2} \|x\|^2$ is convex.
Our main theorem asserts that for every $\beta > 0$,
the family of $\beta$-semi-log-convex densities (with respect to $\gamma_n$)
is uniformly sub-Markovian.

\begin{theorem}\label{thm:main}
Suppose that the non-negative measureable function
$f : \mathbb R^n \to \RR_+$ is $\beta$-semi-log-convex for some $\beta \geq 1$.
Then for all $\alpha \geq e^3$,
\begin{equation}\label{eq:semi}
\gamma_n\left(\{ x \in \mathbb R^n : f(x) > \alpha \|f\|_1\}\vphantom{\bigoplus}\right) \leq \frac{1}{\alpha} \cdot C \sqrt{ \frac{\beta \log \log \alpha}{\log \alpha} }\,,
\end{equation}
where $C > 0$ is a universal constant.
%Let $f : \mathbb R^n \to \mathbb R_+$ be a function with continuous second-order partial derivatives.
\end{theorem}

Note that in proving Theorem~\ref{thm:main}, we may assume (by approximation)
that $f$ has continuous second-order partial derivatives,
and then $\beta$-semi-log-convexity implies that
\begin{equation}\label{eq:hess}
\nabla^2 \log f(x) \succeq -\beta \,\mathrm{Id} \qquad \forall x \in \RR^n\,.
\end{equation}

We first explain how this resolves Conjecture \ref{conj:main} before moving on to
a discussion of Theorem \ref{thm:main}.
Let $\{B_t\}$ be an $n$-dimensional Brownian motion with $B_0=0$, and let $\cP_t f(x) = \E[f(x+B_t)]$
denote the corresponding semigroup.
A proof of the following standard fact is contained in the appendix.

\begin{lemma}\label{lem:hessian}
For any $f : \mathbb R^n \to \mathbb R_+$ in $L^1(\gamma_n)$ and $t > 0$, one has $\nabla^2 \log \cP_t f(x) \succeq -\frac{1}{t} \Id$
for all $x \in \mathbb R^n$.
%For every non-negative function $g(x)$, the function $f(x) \coloneqq P_t (g)$ satisfies the condition \eqref{hesscondition} with $\beta = \frac{1}{t}$.
\end{lemma}

This rather immediately implies the following.

\begin{corollary}\label{cor:main}
There is a constant $C > 0$ such that
for every $\rho \in (0,1)$,
the following holds.
For every measurable $g : \mathbb R^n \to \mathbb R_+$
and every $\alpha \geq e^3$, one has
\begin{equation}\label{eq:conj}
\P\left(\cP_{1-\rho} g(B_{\rho}) > \alpha \|g\|_1\right) \leq \frac{1}{\alpha} \cdot C \sqrt{\frac{\rho}{1-\rho} \frac{\log \log \alpha}{\log \alpha} }\,.
\end{equation}
\end{corollary}

\begin{proof}
%If $\rho=0$, then $P_{1-\rho} g(B_{\rho}) = \|g\|_1$, hence the result is trivial.
%Assume $\rho > 0$.
If we define $f(x)=g(\sqrt{\rho} x)$, then $\cP_{1-\rho} g(B_{\rho})$ and $\cP_{(1-\rho)/\rho} f(Z)$ have the same law,
where $Z$ is a standard $n$-dimensional Gaussian.
Now combining Lemma \ref{lem:hessian} and Theorem \ref{thm:main} yields the desired result.
\end{proof}

Corollary \ref{cor:main} yields a resolution to Conjecture \ref{conj:main} by noting that for any $t > 0$,
\[
\gamma_n\left(\vphantom{\bigoplus}\{x : U_t f(x) >   \alpha \|f\|_1 \}\right) = \P\left(\cP_{1-e^{-2t}} f(B_{e^{-2t}}) > \alpha \|f\|_1\vphantom{\bigoplus}\right)\,.
\]

\medskip
\noindent
{\bf Translating the anti-concentration of Brownian motion.}
Despite the fact that Theorem \ref{thm:main} is not a stochastic statement,
the main theme of our paper is that the variance of Brownian motion
can be translated into anti-concentration estimates
for the information content of semi-log-convex densities on Gaussian space.

Let $f : \mathbb R^n \to \mathbb R_+$ be as in the statement of Theorem \ref{thm:main},
and let us assume that $\int f\,d\gamma_n = 1$.
%Consider
%the Doob martingale $M_t = \cP_{1-t} f(B_t)$ for $t \in [0,1]$.
Our goal \eqref{eq:semi} is equivalent to bounding $\P(f(B_1) > \alpha)$
and to this end,
it will suffice to give an upper bound on $\P(f(B_1) \in [\alpha',2\alpha'])$ for every $\alpha' > \alpha$
(this implies a bound on $\P(f(B_1) > \alpha)$ by a simple dyadic summation).

Very roughly, this will be achieved as follows.  We show that the function $t \to \E[f(B_1) \1_{\{f(B_1) \geq t \}}]$ does not decrease much when $t$ varies between $\alpha'$ and $2 \alpha'$ by introducing a coupling which associates the corresponding level sets for different values of $t$.
This ``transfer of mass'' between levels is achieved by adaptively perturbing an
underlying It\^o process.
The Hessian condition \eqref{eq:hess} ensures that $f$ behaves predictably under small perturbations.
The primary difficulty is to perform the perturbations without changing
the measure of the underlying Brownian motion too much.
For this purpose, Girsanov's change
of measure theorem will play an essential role.

\medskip
\noindent
{\bf Related work.}
Our use of random measures and stochastic calculus to study the geometry of Gaussian space
is certainly closely related to the works \cite{Eldan13,EldanArxiv}.  On the other hand, the idea to
study functionals using an ``optimal'' adapted coupling to Brownian motion (see Section \ref{sec:drift})
comes from the viewpoint of stochastic control theory \cite{Follmer85,Lehec13} and its geometric applications \cite{Lehec13}.
Other variational perspectives appear in the work \cite{BD98} and in Borell's papers \cite{Borell00,Borell02}
where one of his primary goals is their use in proving functional inequalities.
An important distinction between our work and some previous ones involves our use of second-order methods.
Specifically, we study the effect of perturbations on the optimal drift.

Finally, we should mention two vast bodies of work closely related to our study:
Markov diffusions and semigroup methods (see, e.g., \cite{BGLbook}),
as well as the the theory of optimal transportation.  For the latter topic,
one might consult \cite[Ch. 9]{Villani03book} for an excellent review of the
literature
related to functional inequalities.

\subsection{Talagrand's conjecture for the discrete cube}
\label{sec:discrete}

Talagrand \cite{T89} posed the following conjecture which is a generalization of Conjecture \ref{conj:main}.
Let $n \geq 1$ and $t \geq 0$ be given.  Consider the probability measure on the set $\{-1,1\}$ given by
\[
\mu_t \defeq \frac{1-e^{-t}}{2} \delta_{-1} + \frac{1+e^{-t}}{2} \delta_1\,.
\]
Denote $\mu_t^n$ the corresponding product measure on $\{-1,1\}^n$,
and put $\mu = \mu_{\infty}$ so that $\mu^n$ is the uniform measure on $\{-1,1\}^n$.
Let $L^2(\mu^n) = L^2(\{-1,1\}^n,\mu^n)$ denote the Hilbert space of real-valued functions
$f : \{-1,1\}^n \to \mathbb R$.

Consider the operator $T_t : L^2(\mu^n) \to L^2(\mu^n)$ given by convolution
with an $e^{-t}$-biased measure, i.e.
\[
T_t f = f * \mu_t^{n}\,,
\]
where one uses the natural multiplicative group structure on $\{-1,1\}^n$.

As in the Gaussian case, this operator admits a hypercontractive estimate \cite{Bonami71,Beckner75,Gross75}:
$T_t$ is a contraction from $L^p(\mu^n)$ to $L^q(\mu^n)$ for $1 < p \leq q$
and $t \geq \frac12 \log \frac{q-1}{p-1}$.

\begin{conjecture}[\cite{T89}]\label{conj:cube}
For every $t > 0$, there exists a function $\varphi_t : \mathbb [1,\infty) \to [1,\infty)$ with $\lim_{\alpha \to \infty} \varphi_t(\alpha) = \infty$ such
that for every $f : \{-1,1\}^n \to \mathbb R_+$ and any $\alpha > 1$,
\[
\mu^n\left(\vphantom{\bigoplus}\{x \in \{-1,1\}^n : T_t f(x) > \alpha \|f\|_1 \}\right) \leq \frac{1}{\alpha\, \varphi_t(\alpha)}\,.
\]
\end{conjecture}

It is a straightforward observation that Conjecture \ref{conj:cube} implies Conjecture \ref{conj:main} with $\psi_t = \varphi_t$.
This is proved by embedding Gaussian space (approximately) into a sequence of discrete cubes of growing dimension via the central limit theorem; we refer to
the discussion in \cite{BBBOW13}.
At present, Conjecture \ref{conj:cube} is open for any value of $t > 0$.

In his original paper \cite{T89}, Talagrand did provide a proof of a related inequality for the averaged operator $A = \int_0^1 T_t\,dt$.
Specifically, there is a constant $C > 0$ such that for all $\alpha > e^3$,
\[
\mu^n\left(\left\{x : A f(x) \geq \alpha \|f\|_1 \right\}\right) \leq \frac{C \log \log \alpha}{\log \alpha}\,.
\]
His proof makes clever use of $\approx \log \alpha$ invocations of the aforementioned hypercontractive inequality.

\section{Entropy, energy, and the F\"ollmer drift}
\label{sec:drift}

Fix $n \geq 1$ and consider $\mathbb R^n$ with the equipped with the
standard Euclidean structures $\langle\cdot,\cdot\rangle$ and $\|\cdot\|$, and
the Gaussian measure $\gamma_n$ defined by
$$
\frac{d \gamma_n}{dx} = \frac{1}{(2 \pi)^{n/2}} \exp \left ( - \Vert x \Vert^2 / 2 \right ).
$$
We now lay out the basic objects of our study and
prove some preliminary properties.  In the next section,
we begin with an informal discussion highlighting
a stochastic calculus approach to the geometry of Gaussian space.
This is followed by a broad outline of our arguments.
The formal preliminaries begin in Section \ref{sec:prelims} and the main theorem is proved in Section \ref{sec:mainproof},
save for the core technical lemma of the paper to which Section \ref{sec:temp} is devoted.

\subsection{Overview and proof sketch}

Suppose now that $f : \mathbb R^n \to \mathbb R_+$ has continuous second-order partial derivatives and $\int f\,d\gamma_n = 1$.
Recall that, given $\alpha > 0$, we are interested in showing that $\P(f(B_1) \in [\alpha,2\alpha]) \ll 1/\alpha$
as $\alpha \to \infty$, where $\{B_t\}$ is a Brownian motion with $B_0=0$.
Since $f$ could be concentrated on a set of very small measure, this would necessitate the study
of events of very small probability.
Instead, we will restrict our attention to the interesting
parts of the space by changing the measure of the Brownian motion so that $B_1$ has law $f d\gamma_n$.

To this end, we define an It\^o process $\{W_t\}$ by the stochastic differential equation
\[
W_0 = 0,\quad dW_t = dB_t + v_t\,dt
\]
for some predictable drift process $\{v_t\}$ with respect to the filtration $\{\mathcal F_t\}$
underlying the Brownian motion.  Moreover, we will choose this drift as the
solution to an energy optimization problem.

The following variational viewpoint is taken from the papers of F\"ollmer \cite{Follmer85} and Lehec \cite{Lehec13}.
Lehec's work convincingly demonstrates its geometric applicability
and it provided us with considerable inspiration.
Let us take any predictable drift $\{u_t\}_{t \in [0,1]}$ such that $B_1 + \int_0^1 u_t\,dt$ has law $f d\gamma_n$.
Among all such drifts, we will define $\{v_t\}$ to be the one that minimizes the quantity
\[
\frac12 \int_0^1 \E\, \|u_t\|^2\,dt\,.
\]

It is quite beneficial to think of $\{v_t\}$ as the minimum-energy adapted coupling between $d\gamma_n$ and $f d\gamma_n$.
Furthermore, one can connect this energy to the entropy of $f$:
\begin{equation}\label{eq:optimality}
H_{\gamma_n}(f) = \frac12 \int_0^1 \E\, \|v_t\|^2\,dt\,,
\end{equation}
where $H_{\gamma_n}(f) \defeq \int f \log f\,d\gamma_n$ denotes the relative entropy of $f$ with respect to $\gamma_n$.
It turns out that optimality of $v_t$ implies that $\{v_t\}$ is a martingale with respect to
$\{\mathcal F_t\}$, a fact that will be central in our study.
In particular, the martingale property will imply that the behavior of $\{W_t\}$ at small times
must have echoes that reverberate to time $1$.

As we will see below, one can compute explicitly
\begin{equation}\label{eq:drift0}
v_t = \nabla \log \cP_{1-t} f(W_t)\,.
\end{equation}

This has a straightforward geometric interpretation.  Consider the relative density
\[
\phi_t(x) =  f(x) e^{-\|x-W_t\|^2/2(1-t)}\,,
\]
and let $\bar \phi_t(x)$ be the normalization of $\phi_t(x)$ such that $\bar \phi_t(x)\, dx$ is a probability density.  Then,
\[v_t = (1-t)^{-1} \left(\int x \bar \phi_t(x)\,dx - W_t\right)
\]
is the vector pointing from $W_t$ to the center of mass of $f$ with respect to a Gaussian distribution of variance $1-t$ centered at $W_t$.
The scaling by $(1-t)^{-1}$ stands to reason:  The fact that $W_1 \sim f d\gamma_n$ means that as $t$ approaches $1$,
if $W_t$ is far from the ``bulk'' of $f$, the desperation of the drift increases.

With the optimal drift $v_t$, the process $W_t$ has a useful alternative description: Suppose that the process $\{B_t\}$ has the law of a Brownian motion when the underlying probability space is a equipped with a measure $P$ (hence, $P$ is the ``default'' measure under which we have defined the processes above). Then, the law of the process $\{W_t\}$ (on the space of paths) coincides with the law of $\{B_t\}$ under the measure $f(B_1) d P$.

It is possible to show (e.g., using the tools of the next section) that equation \eqref{eq:drift0}
implies that for every $t \in [0,1]$,
\begin{equation}\label{eq:fisher}
\E\, \|v_t\|^2 = \int \frac{\|\nabla \cP_{1-t} f\|^2}{\cP_{1-t} f}\,d\gamma_n\,.
\end{equation}
The latter quantity is the Fisher information of $\cP_{1-t} f$ (see \cite[Ch. II.5]{BGLbook}).  Thus the order of magnitude of $\|v_t\|$
reflects, in a sense, the ``granularity'' of $f$ on different scales. 
Given our discussion so far, it is difficult to avoid stating Lehec's elegant proof \cite{Lehec13}
of the Gaussian log-Sobolev inequality:
\begin{equation}\label{eq:logsob}
H_{\gamma_n}(f) \stackrel{\eqref{eq:optimality}}{=} \frac12 \int_0^1 \E\,\|v_t\|^2\,dt \leq \frac12 \E\,\|v_1\|^2 \stackrel{\eqref{eq:fisher}}{=} \frac12 \int \frac{\|\nabla f\|^2}{f}\,d\gamma_n\,,
\end{equation}
where the only inequality is an immediate consequence of the fact that $v_t$ is a martingale.

\medskip
\noindent
{\bf Changes of measure and gradient ascent.}
Recall again that our goal is to bound the probability
$\P(f(B_1) \in [\alpha,2\alpha]) \ll \frac{1}{\alpha} \textrm{ as } \alpha \to \infty\,.$
To this end, we will study the Doob martingale $\cP_{1-t} f(B_t)$.
As just argued, it will be beneficial to consider instead the process $\cP_{1-t} f(W_t)$.
After the change of measure, it suffices to prove simply that
$$\P(f(W_1) \in [\alpha,2\alpha]) \to 0$$ as $\alpha \to \infty$, uniformly for all functions satisfying the assumptions of the theorem.
(Since $W_1$ has the law $f d\gamma_n$, it holds that $\P(f(W_1) \in [\alpha,2\alpha]) \geq \alpha \cdot \P(f(B_1) \in [\alpha,2\alpha])$.)

Now the story comes together, as It\^o's formula will tell us that our process $\cP_{1-t} f(W_t)$
can be related directly to the drift $\{v_t\}$:  For all $t \in [0,1]$,
\begin{equation}\label{eq:drift1}
\cP_{1-t} f(W_t) = \exp\left(\int_0^t \langle v_s, dB_s\rangle + \frac12 \int_0^t \|v_s\|^2\,ds \right).
\end{equation}

As alluded to in the introduction, we will bound $\P(f(W_1) \in [\alpha,e \alpha])$ by perturbing the process $\{W_t\}$. We will do this by defining a process $W_t^\delta$ that has two properties:
\begin{enumerate}
   \item The measure of $\{W_t\}$ is relatively insensitive to such perturbations (explained below).
   \item With overwhelming probability, $\log f(W_1^\delta) \geq \log f(W_1) + 1$.
\end{enumerate}
Combining these two properties, we would then have
$$
\P(\log f(W_1) > \log \alpha) \stackrel{\mathrm{(ii)}}{\approx} \P(\log f(W_1^\delta) > \log \alpha + 1 ) \stackrel{\mathrm{(i)}}{\approx} \P(\log f(W_1) > \log \alpha + 1)\,,
$$
yielding an upper bound on $\P(f(W_1) \in [\alpha, e\alpha])$, as desired.

The perturbed processes are essentially of the following form.
For a fixed $\delta > 0$, define
\[
W^{\delta}_t \coloneqq W_t + \delta \int_0^t v_s\,ds.
\]

Let us first address property (i),
that the measure of the process is not affected too significantly by the perturbations.
In other words, we want to find a new measure $P'$ on the space of paths such that:
(1) under this measure, the perturbed processes $W_t^\delta$ has the same distribution as that of the process $W_t$ under the original measure $P$ and (2) the density of $P'$ with respect to $P$ will be close to 1 for ``most'' sample paths.

To this end, we will
employ Girsanov's theorem to tell us that for every $\delta > 0$, there is a measure $Q_{\delta}$ under which $W_t^{\delta}$ is a Brownian motion, and furthermore that
\begin{equation}\label{eq:com}
   \frac{dQ_{\delta}}{dQ} = \exp\left(- \delta \int_0^1 \langle v_t, dB_t\rangle - \left(\delta+\frac{\delta^2}{2}\right) \int_0^1 \|v_t\|^2\,dt\right)\,,
\end{equation}
where $Q=Q_0$ is a measure with respect to which the process $W_t$ is a Brownian motion.

This equality
expresses the relative probability of Brownian motion having the sample path $\{W^{\delta}_t : t \in [0,1]\}$ vs. the sample path $\{W_t : t \in [0,1]\}$. We now recall that the law of the process $\{W_t\}$ coincides with the law of $\{B_t\}$ under the measure $f(B_1) d P$, or in other words $\frac{dP}{d Q} = f(W_1)$. By chaining these two factors, we conclude that the
desired measure $P'$ is given by
\[\frac{dP'}{dP} = \frac{d Q_\delta}{d Q} \frac{f(W_1^\delta)}{f(W_1)}\,.\]

From a high-level perspective, the most important factor in \eqref{eq:com} is
$\exp\left(-\delta \int_0^1 \|v_t\|^2\,dt\right)$.
Thus in order to have $\vphantom{\bigoplus} \frac{d Q_\delta}{d Q} \frac{f(W_1^\delta)}{f(W_1)} \approx 1$,
we need this ``loss'' in measure to be almost exactly compensated by a corresponding increase in the value of $f$
for the perturbed process $W_1^{\delta}$.

In other words, in order to accomplish our goal, we need
\begin{equation}\label{eq:approxeq}
   \frac{f(W_1^\delta)}{f(W_1)} \approx \exp\left(\delta \int_0^1 \|v_t\|^2\,dt\right).
\end{equation}
To this end, we now employ the
Hessian condition \eqref{eq:hess} to conclude that since $W_1^{\delta} = W_1 + \delta \int_0^1 v_t\,dt$, we have
\begin{equation}\label{eq:hessimp}
f(W_{1}^{\delta}) \geq f(W_1) \exp\left(\delta \left\langle v_{1}, \int_0^{1} v_t\, dt\right\rangle - \beta \delta^2 \left\|\int_0^{1} v_t \,dt\right\|^2 \right),
\end{equation}
where we have used the fact that $v_1 = \nabla \log f(W_1)$ from \eqref{eq:drift0}. Ignoring $\delta^2$ term in the preceding
exponent (which requires $\delta$ to be small), \eqref{eq:approxeq} necessitates that
\[
\exp\left(\delta \int_0^1 \|v_t\|^2\,dt\right) \approx \exp\left(\delta \left\langle v_{1}, \int_0^{1} v_t\, dt\right\rangle\right).
\]

Now  we use the martingale property of $v_t$, which implies immediately that
\[\E\left[\left\langle v_1, \int_0^1 v_t\,dt\right\rangle\right] = \int_0^1 \E\,\|v_t\|^2\,dt\,.\]

Thus the last issue we need to address in order to establish property (i) is the {\em concentration} of $\langle v_1, \int_0^1 v_t\,dt\rangle$ and how it interacts with the
many details and lower-order terms that we have glossed over. Controlling this presents the bulk of the technical difficulties in the proof to come. 

Next, let us address property (ii). Going back to equation \eqref{eq:hessimp} and assuming
sufficient concentration of the expression $\langle v_1, \int_0^1 v_t\,dt\rangle$,
it follows that the expression $\log f(W_1^\delta) - \log f(W_1)$ is effectively bounded from below by $\delta \int_0^1 \E\,\|v_t\|^2\,dt$. In turn, the latter expression will be bounded from below using \eqref{eq:drift1}, which can be thought of as a ``path-wise'' log-Sobolev inequality (recall \eqref{eq:logsob}):
Neglecting the martingale term, \eqref{eq:drift1} tells us that whenever the value of $f(W_1)$ is large, so is the expression $\int_0^1 \E\,\|v_t\|^2\,dt$.

On first glance, the crucial fact that the change of measure $\frac{d Q_\delta}{d Q}$ gets almost exactly canceled
out by the term $\frac{f(W_1^\delta)}{f(W_1)}$ may look a bit mysterious. Let us try to shed some light on why \eqref{eq:approxeq}
should hold.
A duality argument based on the Gibbs variational principle (explained, for instance, in \cite{Lehec13})
shows that the drift $v_t$ is extremal in the sense that among all adapted drifts $u_t$, it maximizes the expression
\begin{equation*}\label{eq:variational}
\E \left [\log f \left (B_1 + \int_0^1 u_t\, dt \right ) - \frac{1}{2} \int_0^1 \| u_t \|^2\, dt \right ].
\end{equation*}
As a special case, it follows that the function
\begin{equation}\label{eq:deltad}
\delta \mapsto \E \left [\log f \left (W_1^\delta \right ) - \frac{1}{2} (1+ \delta)^2 \int_0^1 \| v_t\|^2\, dt \right ].
\end{equation}
is maximized at $\delta = 0$.

Under mild assumptions on $f$, this function is analytic and thus its derivative vanishes at $\delta = 0$,
implying that
\[
   \E\left[\log \frac{f(W_1^{\delta})}{f(W_1)}\right] = \delta\, \E\left[\int_0^1 \|v_t\|^2\,dt\right] + O(\delta^2)\,.
\]

The bulk of our argument now amounts to showing that the second derivative of \eqref{eq:deltad} with respect to $\delta$ is not too negative whenever $f$ is semi-log-convex.
In other words, for $\delta > 0$ small enough, the drift $\{(1+\delta) v_t\}$ is not too much worse (in terms of minimizing the functional $\{u_t\} \mapsto \E[\int_{0}^1 \|u_t\|^2\,dt]$) than the optimal adapted drift $\{u_t\}$ that achieves law $W_1^{\delta}$.

%A lower bound on the second derivative will allow us to take sufficiently large values of $\delta$ thus attaining a significant increase in the value of $f(W_1^\delta)$. The semi log-convexity of the function $f$ amounts to the fact that the second derivative of $\delta \to \log \frac{f(W_1^\delta)}{f(W_1)}$ is bounded from below by a negative constant depending on $\beta$ and on the expression $\left |\int_0^1 v_t dt \right |$, which can in turn be controlled by introducing a stopping time for the perturbation. 

\subsection{Formal preliminaries}
\label{sec:prelims}
%We have the equalities
%\begin{eqnarray*}
%d( \cP_{1-t} f(B_t) ) &=& \langle \nabla (\cP_{1-t} f)(B_t), dB_t\rangle \\
%d(  \nabla \cP_{1-t} f(B_t) ) &=& \nabla^2 (\cP_{1-t} f)(B_t) dB_t\,.
%\end{eqnarray*}

We fix a non-negative function $f : \mathbb R^n \to \mathbb R_+$
with continuous partial derivatives of second order.
Moreover, we fix a measurable sample space $(\Omega, \Sigma)$ which we assume to be rich enough to support an $n$-dimensional
Brownian motion.

%Moreover, we will use the following notation: for a smooth enough function $\varphi:\RR^n \to \RR$ we denote by $\nabla \varphi$ its %gradient and by $\Delta \varphi$ the usual Laplacian. We fix a probability space $(\Omega, \Sigma)$, and for a random variable $X$ and %a measure $P$ on $(\Omega, \Sigma)$, we denote by $\E_P[X]$ the expectation of $X$ with respect to $P$.

Let $\{W_t : t\in [0,1]\}$ be a process adapted to a filtration $\{\mathcal{F}_t\}$ and let $Q$ be a measure over the sample space $(\Omega, \Sigma)$ such that $W_t$ is a standard $n$-dimensional Brownian motion with respect to $Q$.
Define for all $0 \leq t \leq 1$,
$$
M_t = \cP_{1-t} f(W_t).
$$
Recall that the heat semigroup satisfies
$$
\partial_t \cP_{1-t} f = -\frac12 \Delta \cP_{1-t} f, ~~ \forall 0 < t < 1\,,
$$
and that for all $0 \leq t < 1$, the function $\cP_{1-t} f : \RR^n \to \RR_+$ has continuous derivatives of all orders.
This allows us to apply It\^{o}'s formula (see, e.g., \cite{OxBook}) in order to calculate
\begin{eqnarray}\nonumber
dM_t = d( \cP_{1-t} f(W_t) ) &=& \partial_t \cP_{1-t} f(W_t) dt + \langle \nabla (\cP_{1-t} f)(W_t), dW_t\rangle + \frac12 \Delta (\cP_{1-t} f)(W_t) dt \\
 &=& \langle \nabla (\cP_{1-t} f)(W_t), dW_t\rangle\label{eq:dmt} \\
 &=&
 M_t \langle v_t, dW_t \rangle\,,\nonumber
\end{eqnarray}
where we define
\begin{equation}\label{eq:explicit}
v_t \defeq \frac{\nabla (\cP_{1-t} f)(W_t)}{M_t} = \frac{\nabla (\cP_{1-t} f)(W_t)}{\cP_{1-t} f(W_t)} = \nabla (\log \cP_{1-t} f)(W_t)\,.
\end{equation}
Moreover, by definition of the operator $\cP_{1-t}$ we have $M_t = \E_Q[M_1 | \mathcal{F}_t]$, so that $M_t$ is a martingale under $Q$.

Next, we construct a measure $P$ on $(\Omega, \Sigma)$ using the equation
\begin{equation} \label{dPdQ}
P(A) = \E_Q[ \mathbf{1}_A M_1]
\end{equation}
for every measurable $A \subset \Omega$. We can also formally understand this definition as $\frac{d P}{d Q} = M_1$.

We define an $\mathcal{F}_t$-adapted process $B_t$ by the equation
$$
B_t = W_t - \int_0^t v_s\, ds.$$
In other words, the process $B_t$ is defined by the equations
\begin{equation} \label{eq:defWt}
B_0 = 0, ~~ d W_t = dB_t + v_t\, dt\,.
\end{equation}

The following theorem, which amounts to an application of Girsanov's theorem, immediately follows as a special case of Theorem 2 and Lemma 3 in \cite{Lehec13}.

\begin{theorem} \label{Lehec}
The process $\{B_t : t \in [0,1]\}$ is well-defined. Moreover, this process is an $\mathcal{F}_t$-Brownian motion under the measure $P$. Furthermore, the following assertions hold.
\begin{enumerate}
\item $W_1$ has the law $f d \gamma_n$ under the measure $P$.
\item Almost surely in $P$, $\int_0^t v_t\, dt$ is defined for all $0 \leq t \leq 1$.
\item $\int_0^1 \Vert v_t \Vert^2\, dt < \infty$ almost surely in $P$.
\item $\E_P [\int_0^1 \Vert v_t \Vert^2\, dt ] = 2\, H_{\gamma_n}(f).$
\end{enumerate}
\end{theorem}

%From this point and on, our "default" measure will be $P$ in the sense that we abbreviate $\E[\cdot] = \E_P[\cdot]$ and $\P(\cdot) = %\\cP_P(\cdot)$.

Next, fix $\tau \in [0,1]$, and recall that $M_t$ is a martingale. Using equation \eqref{dPdQ}, we learn that for all $A \in \mathcal{F}_{\tau}$,
$$
P(A) = \E_Q \bigl [\E_Q [ \mathbf{1}_A M_1 \mid \mathcal{F}_\tau] \bigr ] = \E_Q[\mathbf{1}_A M_\tau ].
$$
It follows that $\{W_t : t \in [0,\tau]\}$ has the law of a Brownian motion under the measure $\frac{1}{M_\tau} dP$ and, furthermore, that for any $0 \leq s \leq \tau$, the process $\{W_t - W_s : t \in [s,\tau]\}$ has the law of a Brownian motion under measure $\frac{M_s}{M_\tau} dP$. Thus, we also have that
\begin{equation} \label{dPdQ2}
P(A \mid \mathcal{F}_s) = \E_Q \left . \left [ \mathbf{1}_A \frac{ M_\tau }{M_s} \,\right |  \mathcal{F}_s \right   ]
\end{equation}
for all $A \in \mathcal{F}_\tau$. The next fact will be crucial (and is also observed in \cite{Lehec13}, in somewhat greater generality).

\begin{fact}\label{fact:vt}
The process $\{v_t : t\in [0,1]\}$ is a martingale under the measure $P$.
\end{fact}

To see this, fix some $0 \leq s \leq t \leq 1$. Define $\sigma_t = \nabla \cP_{1-t} f(W_t) = \cP_{1-t} (\nabla f)(W_t)$ (recalling that $f$ is twice-differentiable). Since $W_t$ is a $Q$-Brownian motion, we have
\begin{align}
\E_Q [\sigma_t \mid \mathcal{F}_s] & = \E_Q [ \cP_{1-t} (\nabla f)(W_t) \mid \mathcal{F}_s ] \nonumber \\
& = \E_Q[ \nabla f(W_1) \mid \mathcal{F}_s ] \nonumber
\\ & = \cP_{1-s} (\nabla f) (W_s) \nonumber
\\ & = \sigma_s\,. \label{eq:sigmamartingale}
\end{align}
This yields
$$
\E_P [v_t \mid \mathcal{F}_s]  = \E_P \left . \left [ \frac{\sigma_t}{M_t} \,\right | \mathcal{F}_s \right ]
 \stackrel{\eqref{dPdQ2}}{=} \E_Q \left . \left [ \frac{\sigma_t}{M_s} \,\right | \mathcal{F}_s \right ]
 \stackrel{\eqref{eq:sigmamartingale}}{=} \frac{\sigma_s}{M_s}
 = v_s.
$$
which establishes the fact.

Finally, using It\^{o}'s formula, equation \eqref{eq:dmt} becomes
$$
d \log M_t = \langle v_t, dW_t\rangle - \frac{1}{2} \|v_t\|^2 dt,
$$
yielding the representation
\begin{align}
\cP_{1-t} f(W_t) = M_t &= \exp\left(\int_0^t \langle v_s, dW_s\rangle - \frac12 \int_0^t \|v_s\|^2\,ds\right) \nonumber\\
&= \exp\left(\int_0^t \langle v_s, dB_s\rangle + \frac12 \int_0^t \|v_s\|^2\,ds\right).\label{eq:repone}
\end{align}
A combination of \eqref{dPdQ} with the last equation finally gives
\begin{equation}
\label{eq:change}
dQ = \exp\left(-\int_0^{1} \langle v_t, dB_t\rangle -\frac12 \int_0^{1} \|v_t\|^2\,dt \right) dP = \frac{1}{M_1} dP = \frac{1}{f(W_1)} dP\,.
\end{equation}
Remark that the above equation makes sense because $M_1>0$ almost surely, since $W_1$ is in the support of $f$. 

In the next section, all probabilities and expectations are taken by default with respect to $P$, the law under which the process $\{B_t\}$ is a Brownian motion.  When we refer to another measure $Q$, we will use the notations $\P_Q$ and $\E_Q$.

\subsection{Proof of the Main Theorem}
\label{sec:mainproof}

Consider a measurable function $f : \mathbb R^n \to \mathbb R_+$ with continuous partial derivatives of second order, such that $\int f\,d\gamma_n=1$
and such that for some $\beta > 1$ and all $x \in \mathbb R^n$,
\begin{equation}\label{eq:hesscondition}
\nabla^2 \log f(x) \succeq -\beta\,.
\end{equation}
We will use the processes and measures defined in Section \ref{sec:prelims} (which depend on $f$).

Recall that $W_1$ has the law of $f\,d\gamma_n$.
As stated previously,
it suffices to prove a (uniform) anti-concentration estimate for the information content of $W_1$.
The next lemma constitutes the main technical step of our argument; its proof occupies Section~\ref{sec:temp}.

%In order to establish Theorem~\ref{thm:main}, it suffices to prove
%a bound of the form (recall Claim~\ref{claim:intro}):
%\begin{equation}\label{eq:onescale}
%\gamma_n\left(\{ z \in \mathbb R^n : f(z) \in [\alpha,e \alpha] \}\vphantom{\bigoplus}\right) \leq \frac{1}{\alpha \psi(\alpha)}\qquad\forall \alpha > 1\,,
%\end{equation}
%where $\psi(\alpha) \geq c \sqrt{\frac{\beta \log \alpha}{\log \log \alpha}}$ for some constant $c > 0$.
%
%Now recall that, under the measure $Q$, the process $W_t$ is a Brownian motion.
%Thus for $\alpha > 1$,
%\begin{align}
%\gamma_n\left(\{ z \in \mathbb R^n : f(z) \in [\alpha,e \alpha] \}\vphantom{\bigoplus}\right)
%&= \P_Q\left(f(W_1) \in [\alpha, e\alpha]\vphantom{\bigoplus}\right) \nonumber\\
%&= \E_Q\left[\1_{\{f(W_1) \in [\alpha,e\alpha]\}}\right] \nonumber\\
%&\stackrel{\mathclap{\eqref{eq:change}}}{=} \E_P\left[\frac{1}{f(W_1)} \1_{\{f(W_1) \in [\alpha,e\alpha]\}}\right] \nonumber\\
%&\leq \frac{1}{\alpha}\P\left(f(W_1) \in [\alpha,e \alpha]\right)\,.\label{eq:justW}
%\end{align}

\begin{lemma}\label{lem:mass}
There is a universal constant $C \geq 1$ such that if
$\alpha$ satisfies $\tfrac{\beta \log \log \alpha}{\log \alpha} < C$, then for all $\beta \geq 1$ and
$\eps > \alpha^{-1/64\beta}$,
	\begin{equation*}\label{eq:qgrowth2}
      \P\left(\vphantom{\bigoplus}\!\log f(W_1) \in \left[\log(\alpha), \log(\alpha) + \eps\right]\right) \leq
   20 \eps \sqrt{\frac{\beta \log \log \alpha}{\log \alpha}}.
	\end{equation*}
\end{lemma}

The statement of the preceding lemma was greatly simplified
by suggestions of Lehec after a draft of
this manuscript was initially circulated.
We thank him for his permission in revising our
argument to incorporate some of his ideas.
In particular, we initially obtained a slightly worse
quantitative dependence of $(\log \log \alpha)^4/\sqrt{\log \alpha}$.

Let us note that Lemma \ref{lem:mass} (with $\e=1$) does indeed yield our goal.
Since $W_1$ has the law of $f d\gamma_n$:  For $\alpha$ sufficiently large,
\begin{align*}
   \int \1_{\{f(x) > \alpha\}} \,d\gamma_n(x) 
   &= \sum_{k=0}^{\infty} \int \1_{\{f(x) \in [e^k \alpha, e^{k+1} \alpha)\}}\,d\gamma_n(x)
  \\ 
  &\leq \frac{1}{\alpha}
  \sum_{k=0}^{\infty} e^{-k} \int f(x) \1_{\{f(x) \in [e^k \alpha, e^{k+1}\alpha]\}}d\gamma_n(x) \\
  &=
  \frac{1}{\alpha} \sum_{k=0}^{\infty} e^{-k} \ \P\left(\log f(W_1) \in [k + \log \alpha, k + \log \alpha + 1]\right) \\
  &\leq O(1) \frac{1}{\alpha} \sqrt{\frac{\beta \log \log \alpha}{\log \alpha}}\,.
\end{align*}

\section{Anti-concentration of the information content}

\label{sec:temp}

Our goal is now to prove Lemma \ref{lem:mass}.
Section \ref{sec:setup} sets up an associated family of stochastic processes.
In Section \ref{sec:estimates}, we provide some preliminary estimates, and in
Section \ref{sec:mass} we complete the proof of Lemma \ref{lem:mass}.

\subsection{The perturbations}
\label{sec:setup}

We now couple our process $W_t$ with a family of It\^{o} processes.

Fix $\alpha \geq e^3$ and
define a stopping time
\[
T \coloneqq 1 \wedge \inf \left\{t : \int_0^t \|v_s\|^2\,ds \geq 2 \log \alpha\right\}
\,.
\]
By defintion,
\begin{equation}\label{eq:large2}
T < 1 \implies \int_0^T \|v_t\|^2\,dt = 2 \log \alpha\,.
\end{equation}
Moreover, Jensen's inequality yields
\begin{equation}\label{eq:JensenT}
\left\|\int_0^T v_t\, dt\right\|^2 \leq \int_0^T \|v_t\|^2\,dt \leq 2 \log \alpha\,.
\end{equation}

For $\delta \in \RR$, we define $\{X_t^{\delta} : t \in [0,1]\}$ by
\[
X^{\delta}_t \coloneqq B_t + \int_0^t \left(1+\delta \1_{\{s \leq T\}}\right) v_s\, ds\, = W_t + \delta \int_0^{t \wedge T} v_s\, ds\, .
\]

Next, we would like to argue that Girsanov's formula (see, e.g., \cite[Chapter 6]{LS}) applies so that $\{X_t^{\delta} : t \in [0,1]\}$ has the law of a Brownian motion under the change of measure
\begin{equation}\label{eq:girsanov}
dQ_{\delta} = \exp\left(- \int_0^{1} (1+ \delta \1_{\{t \leq T\}}) \langle v_t, dB_t\rangle - \frac12 \int_0^{1} (1+\delta \1_{\{t \leq T\}})^2 \|v_t\|^2 dt\right) dP.
\end{equation}
To see this, we first notice that by definition of the stopping time $T$, almost surely
$$
\int_0^1 \Vert \delta \1_{\{t \leq T\}}  v_t \Vert^2 dt \leq 2 \delta^2 \log \alpha\,.
$$
It follows that
$$
\E_Q \left [ \exp \left ( \tfrac 1 2 \int_0^1 \Vert \delta \1_{\{t \leq T\}}  v_t \Vert^2\, dt \right  )  \right ] < \infty\,.
$$
In other words, Novikov's condition holds over the measure $Q$ for the drift $\delta \1_{ \{t \leq T \}} v_t$,
so Girsanov's formula is valid.
In particular,
$\{X_t^{\delta} : t \in [0,1]\}$ has the law of a Brownian motion under the change of measure
\begin{align*}
dQ_{\delta} & = \exp\left(- \int_0^{1} \delta \1_{\{t \leq T\}} \langle v_t, dW_t\rangle - \frac12 \int_0^{1} \delta^2 \1_{\{t \leq T\}} \|v_t\|^2\, dt\right) dQ \\
& = \exp\left(- \int_0^{1} \delta \1_{\{t \leq T\}} \langle v_t, d B_t\rangle - \frac12 \int_0^{1} (\delta^2 + 2 \delta) \1_{\{t \leq T\}} \|v_t\|^2\, dt\right) dQ.
\end{align*}
Combining this with the change of measure formula \eqref{eq:change} yields \eqref{eq:girsanov}. An immediate consequence of the latter is the following:

\begin{fact}
For any interval $I \subset \RR$, one has that
\begin{equation}\label{eq:fact3}
\P(\log f(W_1) \in I) = \E \left[f(X_{1}^{\delta}) \frac{d Q_{\delta}}{dP} \1_{\{\log f(X_1^{\delta}) \in I \}} \right].
\end{equation}
\end{fact}

A central component of the proof will be a lower bound the right hand side of the last equation.
From assumption \eqref{eq:hesscondition} (which comes from \eqref{eq:hess} in Theorem \ref{thm:main}),
it follows that for all $z,u \in \mathbb R^n$,
\[
f(z+u) \geq f(z) \exp \left(\langle u, \grad \log f(z) \rangle - \tfrac \beta 2 \|u\|^2 \right)\,.
\]
Combining this with \eqref{eq:explicit} and fact that $X_{1}^{\delta} = W_{1} + \delta \int_0^{T} v_t\, dt$ yields
\begin{equation}\label{eq:grad}
f(X_{1}^{\delta}) \geq f(W_1) \exp\left(\delta \left\langle v_{1}, \int_0^{T} v_t\, dt\right\rangle - \tfrac 1 2 \beta \delta^2 \left\|\int_0^{T} v_t\, dt\right\|^2 \right)\,.
\end{equation}

Finally, recalling \eqref{eq:change} and \eqref{eq:girsanov}, we have the expression
\begin{equation}\label{eq:change2}
f(W_1) \frac{dQ_{\delta}}{d P} = \frac{dQ_{\delta}}{dQ} = \exp\left(- \delta \int_0^T \langle v_t, dB_t\rangle - \left(\delta+\frac{\delta^2}{2}\right) \int_0^T \|v_t\|^2\,dt\right).
\end{equation}

\subsection{Gradients, stopping times, and the change of measure}
\label{sec:estimates}

Let us define now the random variable
\[
   Z\coloneqq\int_0^T \bigl (\left \langle v_{1}-v_t, v_t \right\rangle dt - \langle v_t,dB_t\rangle \bigr )\,,
\]
and for $\lambda,\gamma \geq 0$,
the following two events:
\begin{eqnarray*}
   \mathcal E_{\lambda} &\coloneqq& \left\{ Z \leq -\lambda \right\}\,, \\
   \mathcal B_{\gamma} &\coloneqq& \left\{ \left |\int_0^T \langle v_t, dB_t\rangle \right | \geq \gamma \sqrt{\log \alpha} \right\}\,.
\end{eqnarray*}

The next two lemmas bound the probabilities of these ``bad'' events.
Additionally, the next lemma provides a key estimate on the concentration of the quantity $f(X_1^{\delta}) \frac{d Q_{\delta}}{dP}$ which will be used in conjunction with equation \eqref{eq:fact3}.

\begin{lemma}\label{lem:lambdabound}
For every $\lambda \geq 0$, we have
\[\pr(\mathcal E_{\lambda}) \leq \exp \left ( - \frac{\lambda^2}{12 \beta \log \alpha} \right )\,.
\]
Furthermore, for any measurable event $\mathcal A$ and any $\delta>0$ and $\lambda > 0$, it holds that
\begin{equation}\label{eq:important}
\E \left[f(X_{1}^{\delta}) \frac{d Q_{\delta}}{dP} \1_{\mathcal A} \right]
\geq\exp\left(-3 \beta \delta^2 \log \alpha - \delta \lambda \right) \P \bigl ( \mathcal A \setminus  \mathcal E_{\lambda} \bigr ) .
\end{equation}
\end{lemma}

\begin{proof}
From \eqref{eq:grad}, the following inequality holds $P$-almost surely
\begin{align}
f(X^{\delta}_{1}) & \frac{dQ_{\delta}}{dP} \nonumber \\
&=
f(W_1) \frac{f(X_{1}^{\delta})}{f(W_1)} \frac{d Q_{\delta}}{dP} \nonumber \\
&\stackrel{\mathclap{\eqref{eq:change2} \wedge \eqref{eq:grad}}}{\geq}\ 
 \exp\left(\delta \left\langle v_{1}, \int_0^{T} v_t dt \right\rangle - \tfrac 1 2 \beta \delta^2 \left\|\int_0^{T} v_t dt\right\|^2 - \delta \int_{0}^{T} \langle v_t, d B_t\rangle - \frac{2\delta+\delta^2}{2} \int_0^{T} \|v_t\|^2 dt\right) \nonumber \\
&= \ 
 \exp\bigl (\delta Z \bigr) {\exp\left(- \tfrac 1 2 \beta \delta^2 \left\|\int_0^T v_t dt\right\|^2- \frac{\delta^2}{2} \int_0^T \|v_t\|^2 dt\right)}. \label{eq:large-value}
 \end{align}

Using \eqref{eq:JensenT} and the assumption that $\beta \geq 1$, we can lower bound the second factor in \eqref{eq:large-value}:
\begin{equation}\label{eq:pre-twist}
f(X^{\delta}_{1}) \frac{dQ_{\delta}}{dP} \geq \exp\bigl (\delta Z \bigr) \exp\left(-3 \beta \delta^2 \log \alpha\right).
\end{equation}
Taking expectations and using the fact that $1=\E [f(B_{1})] = \E[f(X^{\delta}_{1}) \frac{dQ_{\delta}}{dP}]$, we conclude that
\begin{equation}\label{eq:sandwich}
  \E \left[ \exp \left( \delta Z \right) \right] \leq \E \left [ f(X_1^\delta) \frac{d Q_\delta}{dP} \exp\left(3 \beta \delta^2 \log \alpha\right) \right ] \leq \exp\left(3 \beta \delta^2 \log \alpha\right).
\end{equation}
Using Markov's inequality now gives for all $\delta < 0$,
$$
\P (Z \leq - \lambda ) = \P \left ( e^{\delta Z} \geq e^{- \delta \lambda} \right ) \leq e^{3 \beta \delta^2 \log \alpha + \delta \lambda}.
$$
The above is true for any $\delta < 0$. Optimizing over $\delta$, we take $\delta = - \frac{\lambda}{6 \beta \log \alpha}$ to attain
$$
\P (Z \leq - \lambda ) \leq \exp \left ( - \frac{\lambda^2}{12 \beta \log \alpha} \right ).
$$
This establishes the first claim of the lemma.

For the second claim, take $\delta > 0$ to be arbitrary.
Multiply on both sides of \eqref{eq:pre-twist} by $\1_{\mathcal A}$ to obtain
\begin{eqnarray*}
\E \left[
f(X^{\delta}_{1}) \frac{dQ_{\delta}}{dP} \1_{\mathcal A}\right]
&\geq& \exp\left(-3 \beta \delta^2 \log \alpha\right) \E[e^{\delta Z} \1_{\mathcal A}] \\
&\geq& \exp\left(-3 \beta \delta^2 \log \alpha\right) \E[e^{\delta Z} \1_{\mathcal A \setminus \mathcal E_\lambda }]  \\
&\geq& \exp\left(-3 \beta \delta^2 \log \alpha\right) e^{- \delta \lambda} \P \bigl ( \mathcal A \setminus  \mathcal E_{\lambda} \bigr ), \\
\end{eqnarray*}
completing the proof.
%\remove{
%Then two applications of Jensen's inequality yield
%\begin{eqnarray*}
%\E[e^{\delta Z'} \mid {\bar{\mathcal A}}]
%&\geq &
%e^{\E[\delta Z' \mid \bar{\mathcal A}]}  \\
%&\stackrel{(\E Z'=0)}{=}&
%\exp\left( \frac{-q}{1-q}\E[\delta Z' \mid \mathcal A]\right) \\
%&\geq &
%\exp\left(\frac{-q}{1-q} \log \E[e^{\delta Z'} \mid \mathcal A]\right) \\
%&= &
%\left(\E[e^{\delta Z'} \mid \mathcal A]\right)^{-q/(1-q)}\,.
%\end{eqnarray*}
%We know that
%\[
%C'(\beta,\delta) \geq \E[e^{\delta Z'}] \geq (1-q) \E[e^{\delta Z'} \mid \bar{\mathcal A}]\,.
%\]
%Combining the two inequalities above yields
%\[
%\E[e^{\delta Z'} \mid \mathcal A] \geq \left(\frac{1-q}{C'(\beta,\delta)}\right)^{(1-q)/q}\,.
%\]
%Now multiplying through by $\1_{\mathcal A}$ in \eqref{eq:twist}, and taking expectations, we see that
%\begin{eqnarray*}
%\E \left[
%f(X^{\delta}_{1}) \frac{dQ_{\delta}}{dP} \1_{\mathcal A}\right]
%&\geq& C'(\beta,\delta)^{-1} \E[e^{\delta Z'} \1_{\mathcal A}] \\
%&=& q C'(\beta,\delta)^{-1} \E[e^{\delta Z'} \mid \mathcal A] \\
%&\geq& q C'(\beta,\delta)^{-1} \left(\frac{1-q}{C'(\beta,\delta)}\right)^{(1-q)/q} \\
%&=&
%q(1-q)^{(1-q)/q} \exp\left(-\frac{3 \beta  \delta^2 \log \alpha}{q}\right) \\
%&\geq &
%q\max\{(1-q),1/e\} \exp\left(-\frac{3 \beta \delta^2 \log \alpha}{q}\right),
%\end{eqnarray*}
%completing the proof.}
\end{proof}

\begin{lemma}\label{lem:gammabound}
For every $\gamma \geq 0$, it holds that
\[\pr(\mathcal B_{\gamma}) \leq 2 e^{-\gamma^2/4}\,.
\]
\end{lemma}

\begin{proof}
Consider the quadratic variation process
$$
V(t) = \int_0^t  \|v_s\|^2 ds.
$$
According to the theorem of Dambis and Dubins-Schwartz (see, e.g., \cite[Chapter V, Theorem 1.10]{RY}), the process
$$
S(t) \coloneqq \int_0^{V^{-1}(t)} \langle v_s, dB_s\rangle
$$
is a Brownian motion up to the stopping time $\tau = V(1)$.

Using Doob's theorem (e.g., \cite[Chapter II, Theorem 1.7]{RY}) and a standard Gaussian tail estimate (e.g., \cite[Chapter II, Proposition 1.8]{RY}), we have
\begin{equation}\label{eq:tail}
\P\left( F_\gamma \right ) \leq 2 e^{-\gamma^2/4}, \,
\end{equation}
where
$$
F_\gamma \coloneqq \left \{ \max_{t \in [0, 2 \log \alpha]} |S(t)| \geq \gamma \sqrt{\log \alpha} \right \}.
$$
By definition of the stopping time $T$, it holds that $V(T) \leq 2\log \alpha$,
thus
\[
\pr\left(|S(V(T))| \geq \gamma \sqrt{\log \alpha} \right) \leq \pr(F_{\gamma})\,,
\]
completing the proof in light of \eqref{eq:tail}.
\end{proof}

\subsection{Expansion of the level sets}
\label{sec:mass}

We now establish Lemma~\ref{lem:mass}.
Recall that the goal is to show that for all $\alpha$ sufficiently large and $\e > \alpha^{-1/64\beta}$, 
\[
\P\left(\log f(W_1) > \e + \log \alpha\right)
\geq \P\left(\log f(W_1) \geq \log \alpha\right) - 20 \e \sqrt{\frac{\beta \log \log \alpha}{\log \alpha}}\,.
\]
\begin{proof}[Proof of Lemma~\ref{lem:mass}]
Define the event $\mathcal A = \{ \log f(W_1) \geq \log \alpha \}$ and put
\[
\mathcal G \defeq \mathcal A \setminus \{ \mathcal E_{\lambda} \cup \mathcal B_{\gamma}\}
\]
where $\gamma \coloneqq \frac{1}{4} \sqrt{\log \alpha}$ and the value of $\lambda$ will be specified shortly. An application of Lemma \ref{lem:lambdabound} ensures that $\P(\mathcal E_{\lambda}) \leq \exp \left ( - \frac{\lambda^2}{12 \beta \log \alpha} \right )$ and an application of Lemma \ref{lem:gammabound} ensures that $\P(\mathcal B_{\gamma}) \leq 2 \exp(-\gamma^2/4)$. A union bound then yields
\begin{equation}\label{eq:remaining}
\P(\mathcal G) \geq \P(\mathcal A) - \exp \left ( - \frac{\lambda^2}{12 \beta \log \alpha} \right ) - 2 \exp(-\gamma^2/4).
\end{equation}
Our first objective is to show that under suitable assumptions on the parameters $\delta, \lambda$ one has the implication
\begin{equation}\label{eq:gimp}
\mathcal G \mbox { holds } \implies \log f(X_1^{\delta}) > \e + \log \alpha\,.
\end{equation}

To this end, and in light of the gradient estimate \eqref{eq:grad}, one
would like to bound the quantity $\int_0^T \langle v_1, v_t \rangle dt$ from below. Recall that \eqref{eq:repone} implies
\begin{equation}\label{eq:large}
\log f(W_1) = \int_0^1 \langle v_t, dB_t\rangle + \frac12 \int_0^1 \|v_t\|^2\, dt \,.
\end{equation}

Thus conditioned on the event $\overline{\mathcal B_\gamma \cup \mathcal E_\lambda}$,
it holds that
\begin{align*}
\int_0^T \langle v_1, v_t \rangle dt ~& = \quad \int_0^T \langle v_1-v_t, v_t \rangle dt + \int_0^T \|v_t\|^2 dt \\
&= \quad Z + \int_0^T \|v_t\|^2 dt + \int_0^T \langle v_t, d B_t \rangle \\
& \stackrel{\mathclap{\eqref{eq:large} \wedge \eqref{eq:large2}}}{=} \quad Z + \1_{ \{T < 1 \}} \left ( 2 \log \alpha + \int_0^T \langle v_t, d B_t \rangle \right ) + \1_{\{T=1\}} \log f(W_1) \\
& > \quad -\lambda + \min \bigl(\log \alpha, \log f(W_1) \bigl) -\gamma \sqrt{\log \alpha}\,.
\end{align*}
We conclude that, under the assumptions
\begin{equation}\label{eq:assump}
\lambda \leq \frac{1}{4} \log \alpha, ~~ \gamma \leq \frac{1}{4} \sqrt{\log \alpha}\,,
\end{equation}
one has the implication
\begin{equation} \label{eq:gT}
\mathcal G \mbox { holds } \implies \int_0^T \langle v_1, v_t\rangle \,dt \geq \frac{\log \alpha}{2}\,.
\end{equation}

Combining this with the gradient estimate \eqref{eq:grad}, 
one sees that
if $\mathcal G$ holds then
\begin{align*}
f(X_1^{\delta}) &\geq f(W_1) \exp\left(\delta \left\langle v_{1}, \int_0^{T} v_t\, dt\right\rangle - \beta \delta^2 \left\|\int_0^{T} v_t dt\right\|^2 \right)
\\
&\stackrel{\mathclap{\eqref{eq:JensenT}}}
   \geq \alpha \exp\left(\vphantom{\bigoplus} \tfrac{1}{2} \delta \log \alpha - \beta \delta^2 \log \alpha  \right).
\end{align*}

Under the additional assumption
\begin{equation}\label{eq:assump2}
\tfrac{1}{2} \delta \log \alpha -  2 \beta \delta^2 \log \alpha \geq \eps\,,
\end{equation}
this establishes \eqref{eq:gimp}. 

We can therefore write
\[
\P\left(\log f(W_1) > \e + \log \alpha \right)
   \stackrel{\eqref{eq:fact3}}{=} \E \left[f(X_{1}^{\delta}) \frac{d Q_{\delta}}{dP} \1_{\{\log f(X_1^{\delta}) > \log \alpha + \eps \} } \right ]
   \stackrel{\eqref{eq:gimp}}{\geq} \E \left[f(X_{1}^{\delta}) \frac{d Q_{\delta}}{dP} \1_{\mathcal{G}  } \right ] \,.
\]
Finally, we invoke \eqref{eq:important}, which gives
\begin{align}
\P(\log f(W_1) > \e + \log \alpha) ~
   & \geq \exp \left ( - 3\beta \delta^2 \log \alpha - \delta \lambda \right ) \P( \mathcal{G} ) \nonumber\\
   & \stackrel{\mathclap{\eqref{eq:remaining}}}{\geq} \exp \left ( - 3\beta \delta^2 \log \alpha - \delta \lambda \right ) \P( \mathcal{A}) - \exp \left ( - \frac{\lambda^2}{12 \beta \log \alpha} \right ) - 2 \exp(-\gamma^2/4) \nonumber \\
   & \geq \P( \mathcal{A}) - 3\beta \delta^2 \log \alpha - \delta \lambda - \exp \left ( - \frac{\lambda^2}{12 \beta \log \alpha} \right ) - 2 \exp(-\gamma^2/4).\label{eq:almost}
\end{align}

Now choose:
\begin{align*}
   \lambda &\coloneqq \sqrt{12 \beta \log \alpha (\log \tfrac{1}{\e} + \log \log \alpha)}\,, \\
   \delta &\coloneqq \frac{4 \eps}{\log \alpha}\,.
\end{align*}
Using the assumption $\tfrac{12 \beta \log \log \alpha}{\log \alpha} < \frac{1}{16}$, it is straightforward to verify that these choices satisfy assumptions \eqref{eq:assump} and \eqref{eq:assump2}.
Recalling \eqref{eq:almost}, we have
$$
\P(\log f(W_1) > \e + \log \alpha) \geq \P( \mathcal{A}) - 48 \e^2 \frac{\beta}{\log \alpha} - 4 \e \sqrt{\frac{12 \beta (\log \frac{1}{\e}+\log \log \alpha)}{\log \alpha}} - \frac{\e}{\log \alpha} - 2 \alpha^{-1/64}.
$$

By assuming that
$\alpha$ is sufficiently large, this implies that for $\e > \alpha^{-1/64\beta}$,
$$
\P(\log f(W_1) \geq \log \alpha + \eps) > \P(\mathcal A) - 20 \e \sqrt{\frac{\beta \log \log \alpha}{\log \alpha}}\,,
$$
and thus completes the proof.
\end{proof}

\subsection*{Acknowledgements}

The authors are grateful to an anonymous referee for greatly
improving the presentation of the paper. His or her suggestions led to significant simplications in our arguments.
Some of the revisions partially follow ideas of Lehec that arose
from his reading of initial drafts of our manuscript.
We are extremely grateful for his permission to incorporate
those ideas into the proof.

The authors would like to thank Elchanan Mossel for encouraging their collaboration, and
Joseph Lehec for a careful reading of many early drafts of this manuscript,
as well as numerous insightful comments.
Both authors were partially supported by NSF grants CCF-1217256 and CCF-1407779.

Much of this work was done
at Microsoft Research in Redmond.
We are also grateful to the Simons Institute in Berkeley for hosting us
during the final portion of the project.
R.E. acknowledges Jian Ding for introducing him to Talagrand's problem, and
J.L. thanks Ryan O'Donnell for enlightening preliminary discussions.

\bibliographystyle{alpha}
\bibliography{convolution}

\appendix

\section{Proof of Lemma \ref{lem:hessian}}

\begin{proof}
The proof is a simple application of the fact that a mixture of log-convex densities is log-convex (see, e.g., \cite[p.649]{MOA}).
Observe that for any $y \in \RR^n$, the function
$$
x \to \frac{|x|^2}{2 t} + \log \left (\cP_{t} (\delta_y) (x) \right )
$$
is convex (here, $\delta_y$ denotes a Dirac mass supported on $\{y\}$). We now apply the
aforementioned fact to conclude that for any integrable function $g:\RR^n \to [0, \infty)$, the function
$$
x \to \frac{|x|^2}{2 t} + \log  \left ( \int_{\RR^n} g(y) \bigl  ( \cP_{t} (\delta_y) (x) \bigr ) dy \right )
$$
must also be convex. In other words, the function
$$
x \to \frac{|x|^2}{2 t} + \log \cP_{t}(g)
$$
is convex. We conclude that
$$
\nabla^2 \log \cP_{t}(g) \succeq - \nabla^2 \left ( \frac{|x|^2}{2 t} \right ) =  -\frac{1}{t} \mathrm{Id}.
$$
\end{proof}

\end{document}